\def\VERSION{15.12.2020: 17:00}
\def\WHO{Ulisse} 
\def\users{us}    
\def\users{world} 
\documentclass[12pt]{article} 
\usepackage{color}
\usepackage{amsmath}
\usepackage{amsthm}
\usepackage{amssymb}
\usepackage{epsfig}
\usepackage{psfrag}
\usepackage{graphicx}
\usepackage{textcomp}
\usepackage{pgf}
\numberwithin{equation}{section}
\usepackage{upgreek} 
\newtheorem{theorem}{Theorem}[section]

\newtheorem{definition}[theorem]{Definition}

\newtheorem{remark}[theorem]{Remark}

\usepackage{mathrsfs,cite}
\usepackage{ifthen}
\usepackage{ulem}
\usepackage[colorlinks=true, pdfstartview=FitV, linkcolor=blue, 
            citecolor=blue, urlcolor=blue]{hyperref}
\usepackage{cancel}\ifthenelse{\equal{\users}{world}}
{
\newcommand{\REM}[1]{}

	\newcommand{\REPLACE}[2]{#2}
	\newcommand{\INSERT}[1]{#1}
	\newcommand{\DELETE}[1]{}
	\newcommand{\CHECK}[1]{#1}
 	
        \newcommand{\COMMENT}[1]{}
        \newcommand{\TCOMMENT}[1]{}
\newcommand{\UUU}{\color{black}}
\newcommand{\EEE}{\color{black}}
\newcommand{\AAA}{\color{black}}
\newcommand{\TTT}{\color{black}}

    \newcommand{\MARGINOTE}[1]{}
}	
{

\usepackage[notcite,notref,color]{showkeys}
\usepackage{showkeys}
\definecolor{brown}{rgb}{0.6,0.2,0.2}
\newcommand{\REM}[1]{\marginpar{\bfseries\tiny{\color{blue}#1}}}

\newcommand{\REPLACE}[2]{{\color{brown}\cancel{#1}\,\uwave{#2}\color{black}}}
 \newcommand{\INSERT}[1]{{\color{brown}\uwave{#1}\color{black}}}
 \newcommand{\COMMENT}[1]{{\color{blue}\uuline{#1}\color{black}}} 
 \newcommand{\DELETE}[1]{{\color{brown}\cancel{#1}\color{black}}}
 \newcommand{\CHECK}[1]{{\color{brown}\uwave{#1}\color{black}}}

 \newcommand{\TCOMMENT}[1]{{\color{blue}{ #1}}}
\newcommand{\AAA}{\color{blue}}
\newcommand{\TTT}{\color{red}}

\newcommand{\MARGINOTE}[1]{\marginpar{\color{red}\tiny\texttt{#1}}}
\usepackage{fancyhdr}
\pagestyle{fancy}
\headheight=28pt\headwidth=17cm
\definecolor{gray}{gray}{0.5}
\newcount\hour \newcount\minute
\hour=\time
\divide \hour by 60
\minute=\time
\loop \ifnum \minute > 59 \advance \minute by -60 \repeat

\rhead{\color{gray}creep gradient\\
by E.Davoli \& T.Roub\'\i\v cek \& U.Stefanelli}
\chead{}
\lhead{Version \VERSION, file: \jobname.tex, \underline{\Large\color{red}\WHO's working on} \\
compiled:
\number\day.\number\month.\number\year\ at
\the\hour:\ifnum\minute<10 0\fi\the\minute\ h }
}

\newcommand{\R}{\mathbb{R}}
\newcommand{\N}{\mathbb{N}}

\newcommand\DT[1]{\mathchoice
                 {{\buildrel{\hspace*{.1em}\text{\LARGE.}}\over{#1}}}
                 {{\buildrel{\hspace*{.1em}\text{\Large.}}\over{#1}}}
                 {{\buildrel{\hspace*{.1em}\text{\large.}}\over{#1}}}
                 {{\buildrel{\hspace*{.1em}\text{\large.}}\over{#1}}}}
\newcommand\DDT[1]{\mathchoice
   {{\buildrel{\hspace*{.13em}\text{\LARGE.\hspace*{-.13em}.}}\over{#1}}}
   {{\buildrel{\hspace*{.1em}\text{\Large.\hspace*{-.1em}.}}\over{#1}}}
   {{\buildrel{\hspace*{.1em}\text{\large.\hspace*{-.1em}.}}\over{#1}}}
   {{\buildrel{\hspace*{.1em}\text{\large.\hspace*{-.1em}.}}\over{#1}}}}

\renewcommand{\d}{{\rm d}}
\newcommand\FE{\varphi_{_{\rm E}}}
\newcommand\FH{\varphi_{_{\rm H}}}
\newcommand\FG{\varphi_{_{\rm G}}}
\newcommand\PP{\varPi}
\newcommand\Cof{\mathrm{Cof}}

\newcommand{\nablaS}{\nabla_{\scriptscriptstyle\textrm{\hspace*{-.3em}S}}^{}}
\newcommand{\divS}{\mathrm{div}_{\scriptscriptstyle\textrm{\hspace*{-.1em}S}}^{}}
\newcommand{\pG}{p_{\rm G}}

\newcommand{\DD}{{\rm D}}

\begin{document}

\MARGINOTE{possible~journals:
\\
\\
\\or\\
MMS
\\
or
\\
SIMA
\\
or
\\Eur. J. Mech. A/Solids
\\
or
\\
.......???}

\noindent{\LARGE\bf
A note about hardening-free viscoelastic\\[.2em] models
in Maxwellian-type \DELETE{creep }rheologies \\[.2em] at large strains}

\bigskip\bigskip

\noindent{\large\sc  Elisa Davoli}\\
{\it Institute of Analysis and Scientific Computing,
TU Wien,\\Wiedner Hauptstrasse 8--10, \UUU A-1040  Vienna, Austria
}

\bigskip

\noindent{\large\sc Tom\'{a}\v{s} Roub\'\i\v{c}ek}\\
{\it Mathematical Institute, Charles University, \\Sokolovsk\'a 83,
CZ--186~75~Praha~8,  Czech Republic
}\\and\\
{\it Institute of Thermomechanics, Czech Academy of Sciences,\\Dolej\v skova~5,
CZ--182~08 Praha 8,  Czech Republic
}

\bigskip

\noindent{\large\sc Ulisse Stefanelli}\\
{\it Faculty of Mathematics, University of
  Vienna,\\Oskar-Morgenstern-Platz 1, A-1090 Vienna, Austria,
}\\and\\
{\it Vienna Research Platform on Accelerating
  Photoreaction Discovery, University of Vienna,\\W\"ahringerstra\ss e
  17,  A-1090 Vienna,   Austria,}
  \\and\\
{\it 
 Istituto di
  Matematica Applicata e Tecnologie Informatiche {\it E. Magenes},\\via
  Ferrata 1, I-27100 Pavia, Italy.
 }


\bigskip\bigskip

\begin{center}\begin{minipage}[t]{16.5cm}

{\small

\noindent{\bfseries Abstract.}\baselineskip=12pt
 {} Maxwellian-type  rheological models of inelastic effects of
 creep type  at large strains are
revisited 
 in relation to 
inelastic-strain gradient\REPLACE{-terms}{ theories}. In particular,
we \UUU observe \EEE
that a dependence of the stored-energy
density on 
inelastic-strain gradients \UUU may lead \EEE to 
spurious hardening effects, \UUU preventing the model from \EEE 
accommodating large inelastic slips. The main  result  of this paper is
an alternative  inelastic model of creep type,  \UUU where \EEE higher-order energy-contribution
is provided by the  gradients  of the elastic strain\INSERT{ and of
the plastic strain rate}, thus  preventing the onset of spurious
hardening \UUU under large \EEE slips. \EEE  
 The combination of 
Kelvin-Voigt
damping and Maxwellian creep  results in a  Jeffreys-type  rheological model.
Existence of weak solutions is  proved  via  \UUU a \EEE Faedo-Galerkin approximation.


\medskip


\medskip


} 
\end{minipage}
\end{center}

\bigskip

\section{Introduction}

Inelasticity at large strain has been the focus of an intense research activity for decades, first from the
 engineering community,  see, e.g.,  the monographs
\cite{DunPet05ICP,JirBaz02IAS,Maug92TPF}, and subsequently also from the mathematical point of view (see, e.g., the recent contributions \cite{davoli-kruzik-pelech, melching-neunteufel-schoeberl-stefanelli, kruzik-melching-stefanelli} on large-strain rate-independent processes, incomplete damage, and finite plasticity, respectively as well as the monographs \cite{MieRou15RIST,KruRou19MMCM} and the references therein).

Within the mathematical purview, there is a general agreement that the
rigorous analysis of large-strain inelastic time-evolving phenomena
requires higher-order regularizations of the inelastic strains 
\cite{DavFra15CRFE,GraSte17FPQE,MaiMie09GERI,Miel02FELG,MiRoSa18GERV,MieRou15RIST,MieRou16RIEF,RouSte19FTCS}. 
Existence theories without gradient regularization are available only
in one space dimension \cite{MelSte??WPON}, at the incremental level
\cite{Mielke04,MielkeMueller,compos}, or under stringent modeling
restrictions \cite{kruzik-melching-stefanelli,Mielke04}. 
In  the  engineering literature, on the other hand, gradient
theories  at
large strains are seldom considered, see \cite{Bett05CM,DunPet05ICP,NauAlt07MCSA},
\cite[Ch.~25]{JirBaz02IAS}, \cite[Ch.~8]{Maug92TPF}, and existence of
solutions not in focus.

Gradient theories for the inelastic strain 
introduce  an  internal length-scale in the problem 
related to  the characteristic width
of inelastic slip-bands arising  during creep, damage, or
plastification processes.  The occurrence of such scale is however
not expected to  cause  additional hardening.  
Although sometimes
strain or time hardening are to be considered
\cite{Bett05CM,NauAlt07MCSA}, in many
applications,  inelastic  models are 
ultimately desired not to exhibit any hardening effect during
long-lasting slip deformations. 
In metals, for example, very large irreversible plastification can occur within the phenomenon
sometimes referred to as {\it superplasticity}.  Large slips with
no hardening are particularly common in  
rock, soil, or ice mechanics. Typically, the slip
on tectonic faults can easily accommodate kilometers during millions of years.
 Glaciers flow kilometers,  with hardening only  occurring 
at temperatures  below  $-70^\circ$\,C 
\cite{SchDuv09CFI}.  In a very different context, large
deformations without hardening can be observed in polymers as
well.  

As a result,  one is interested in identifying inelastic
strain-gradient modelizations   guaranteeing, on the one hand,
that the existence of time-evolution of inelastic phenomena is
mathematically well-posed, and on the other hand, that no 
spurious  hardening effects are  generated. The focus of this
paper is hence on 
introducing a novel hardening-free inelastic model of creep-type
allowing for existence of solutions. In order to accomplish this, the
energy of the medium is assumed to contain a term depending on the
gradient of the {\it elastic} strain. This contrasts with usual
approaches based on {\it total} strain-gradient or {\it inelastic}
strain-gradient regularization. Indeed, we \UUU present an example in
Subsection \ref{sec-motiv} below \EEE showing the \AAA possible \EEE effect of such usual
strain-gradient \UUU regularizations \EEE on the onset of spurious hardening. 

 Our new  model is \UUU introduced \EEE in Section~\ref{sec-model}. In
addition to elastic-strain hardening, we assume \UUU the \EEE viscous dissipation to
be quadratic and \UUU to depend \EEE on the gradient of the inelastic-strain
rate. This last gradient term does not affect the hardening-free
nature of the model. 

Eventually, Section~\ref{sec-anal} focuses on the 
existence of weak solutions to the model. The proof relies on a
Faedo-Galerkin approximation,  as well as \UUU on \EEE compactness, and lower semicontinuity
arguments. 
%



\section{A hardening-free viscoelastic model}
\label{sec-model}

\UUU We devote this section to introducing and commenting our modeling
choices. \EEE

Following the classical mathematical theory of inelasticity at 
large  strains  \cite{gurtin,hill, lubliner1}, we assume that the elastic behavior of our
specimen $\varOmega\subset \R^d$, $d=2,3$,
is independent from preexistent inelastic distortions. This can be
rephrased as the assumption that the deformation gradient $F:=\nabla
y$ associated to any deformation $y:\varOmega\to \R^d$ of the body
decomposes into an elastic strain and an inelastic one. For linearized
theories, this decomposition would have an additive nature; in the
setting of large-strain inelasticity, instead, this behavior is
traditionally modeled via a multiplicative decomposition. In the
mathematical literature different constitutive models have been taken
into account, see, e.g., \cite{DavFra15CRFE, GraSte17FPCM,
  GraSte17FPQE, naghdi}  in   the framework of finite
plasticity. We focus here on the classical  multiplicative
decomposition ansatz  \cite{Kron60AKVE,LeeLiu67FSEP}, recently justified in the setting of dislocation systems and crystal plasticity in \cite{conti.reina,conti.reina2}), in which deformations $y\in H^1(\varOmega;\R^d)$ fulfill 
\begin{align}\label{split}
F=F_{\rm el}\PP,
\end{align}
where $F_{\rm el}$ and $\PP$ denote the elastic and inelastic strains,
respectively.

\UUU \subsection{Tensorial notation}\EEE

In the following, we use capital letters to indicate tensors and
tensor-valued functions, independently from their dimensions. For
$A,\, \widehat A, \, \widetilde A
\in \R^{d\times d}$, $B, \,\widehat B\in \R^{d\times d\times d}$,
and  $C,\, \widehat C\in  \R^{d\times d\times d\times d}$ we use the
standard notation  for contractions on two, three, and four indices,
namely, 
\begin{align*}
&A{:}\widehat A = A_{ij} \widehat A_{ij}, \   B{\vdots} \widehat B
= B_{ijk}\widehat B_{ijk},\   (C{:}A)_{ij} = C_{ijkl}A_{kl},\  
(B{:}A)_i =
B_{ijk}A_{jk},\  C{:}{:}\widehat C = C_{ijkl}\widehat C_{ijkl}
\end{align*}
(summation convention
over repeated indices). On the other hand, contraction on one index will
be marked by $\cdot$ only in case of vectors. In particular,
$(CA)_{ijkl}=C_{ijkm}A_{ml}$, $(BA)_{ijk} = B_{ijm}A_{mk}$, etc. 
The symbol $\top$
indicate transposition of two-tensors, namely $A^\top_{ij}=A_{ji}$,
whereas we denote by the superscript ${\rm t}$ the partial transposition of a
four-tensor with respect to the first two indices, namely $C^{\rm
  t}_{ijkl} = C_{jikl}$. For $A \in  \R^{d\times d}$ we
indicate its symmetric part by ${\rm sym}\, A = (A+A^\top)/2$ and, if
$A$ is invertible, use
the shorthand notation $A^{-\top} = (A^{-1})^\top$. We will use the algebra
$A\widehat A{:}\widetilde A=A{:}\widetilde A\widehat A^\top$ and $A{:}\widehat A \widetilde A = \widehat
A^\top A{:}\widetilde A$. 

Let us recall that, for a differentiable function $F: {\mathbb
  R}^{d\times d} \to \R^{d\times d}$ and $A, \, \widehat A \in {\mathbb
  R}^{d\times d} $ we have that $\DD F(A) \in {\mathbb
  R}^{d\times d\times d\times d}$ and $\DD F(A){:}\widehat A = ({\rm d}/{\rm d}
\alpha) F(A +\alpha \widehat A)|_{\alpha=0}$. In particular, one has that
$\DD(A^{-1}){:}\widehat A = - A^{-1}\widehat A A^{-1}$. Moreover, one easily
checks that $\DD(F^\top) = (\DD F)^{\rm t}$, so that one has  that 
$\DD (A^{-\top}){:}\widehat A = - A^{-\top}\widehat A^{\top} A^{-\top}$.
 Given two  other 
differentiable functions $\widehat F: {\mathbb
  R}^{d\times d} \to \R^{d\times d}$ and $f: {\mathbb
  R}^{d\times d}\to \R$ one has  that 
${\rm D} (f\circ F)(A){:}\widehat A = {\rm D} f(F(A)) {:} \DD F(A) {:}\widehat A$ and 
$\DD(\widehat F \circ F)(A) {:}\widehat A =\DD F(\widehat F(A)) {:} \DD \widehat F(A) {:}
\widehat A$.

Let the reference domain $\varOmega \subset \R^d$ be open  and
 with Lipschitz boundary $\varGamma$, and  let $n$ be the 
outward-pointing unit
normal vector at the boundary. For a $m$-tensor valued function $x \in \varOmega
\mapsto A(x)\in (\R^{d})^m $ with $m\geq 1$  
we define the gradient $\nabla A(x) \in  (\R^{d})^{m+1} $ and the
divergence ${\rm div} A(x) \in  (\R^{d})^{m-1} $
componentwise as  
$$\nabla A(x)_{i_1\dots i_m j} = \frac{\partial}{\partial x_j}
A_{i_1\dots i_m}(x),\quad ({\rm div}
A(x))_{i_1\dots i_{m-1}} =  \sum_{j=1}^d\frac{\partial}{\partial
  x_j}A(x)_{i_1\dots i_{m-1} j} .$$ 
For all $x\in\varOmega\mapsto A(x)\in \R^{d\times d}$
and $x\in \varOmega\mapsto\widehat A(x)\in \R^{d\times d}$ we
have that $
  \nabla (A\widehat A) = (\widehat A^\top\nabla A^\top)^{\rm t} + A \nabla
  \widehat A$.
Let now    $x \in \varOmega
\mapsto v(x)\in \R^{d}$, $x \in \varOmega
\mapsto A(x)\in \R^{d\times d}$, and $x \in \varOmega
\mapsto B(x)\in \R^{d\times d \times d}$ be given. Under suitable regularity
assumptions the following Green formulas can be checked 
\begin{subequations}\label{green}
  \begin{align}
    & \int_\varOmega A {:} \nabla v \, \d x = - \int_\varOmega {\rm div} A
    {\cdot} v
    \, \d x + \int_\varGamma (An) {\cdot} v\, \d x\,, \label{eq:green1}\\
    &\int_\varOmega B {\vdots} \nabla A \, \d x = - \int_\varOmega A{:}{\rm
      div} B \, \d x + \int_\varGamma (A{:} B){\cdot} n \,\d x\,. \label{eq:green2}
  \end{align}
\end{subequations}
Eventually, let $\divS$ denote the $(d{-}1)$-dimensional surface divergence
on $\varGamma$. For vector-valued functions $x\mapsto v(x)\in \R^d $ this is defined as
$$\divS v= {\rm tr}\nablaS v \ \ \text{for} \ \  \nablaS v:= \nabla v - \frac{\partial v}{\partial n} \otimes
n\,,$$ 
where ${\rm tr}$ stands for the trace. The same definition will be
used row-wise for tensor-valued functions. We will use the 
\cite[Formula (34)]{Fried-Gurtin06}
\begin{equation}\label{greenS}\int_\varGamma A{:}\nablaS v\, \d S = - \int_\varGamma (\divS
A{\cdot}v + 2\mathfrak{h} A n {\cdot}v) \, \d S\,,
\end{equation}
where $\mathfrak{h}$ stands for the mean curvature of $\varGamma$.
Arguing row-wise, an analogous relation can be checked to hold for tensors-valued
functions as well.

\UUU \subsection{Stored energy} \EEE

Our aim is that of introducing  a hardening-free  inelastic 
model. 
In absence of hardening, the mathematical analysis of
inelastic evolution is notoriously challenging. In order to make the
existence of weak solutions amenable, we include in the model 
higher-order (gradient) effects. 
More specifically, we define
\begin{align}
&\varPhi(y,\PP)=\int_\varOmega\FE(\nabla y\,\PP^{-1}
)
+\FH(\PP)+\FG(\nabla(\nabla y\,\PP^{-1}))
\, \d x\,.
\label{free-energy+}
\end{align}
\UUU Here, \EEE $\FE:\R^{d\times d}
\to [0,\infty)$ \UUU corresponds to \EEE the elastic energy density
\UUU of the medium and will be assumed to be coercive and to control
the sign of ${\rm det} \,F_{\rm el}$, see \eqref{ass-FM} below. On the
other hand, \EEE 
 $\FH:\R^{d\times d}\to [0,\infty]$ plays the role of a constraint
 \UUU on 
${\rm det}\,\PP$. \EEE
\UUU In particular, we are interested in \EEE 
choices of $\FH$ enforcing
the usual isochoric constraint ${\rm det}\PP=1$ in an approximate sense and
keeping ${\rm det}\PP$ away from negative values, \UUU see
\eqref{ass-plast-large-HD-growth} below. \EEE An explicit example
for  such a term is 
\begin{align}
\FH(\PP):=\begin{cases}\displaystyle{
\frac{\delta}{\max(1,\det\PP)^r}
+\frac{(\det\PP-1)^2}{2\delta}}\!\!&\text{ if }\ \det\PP>0,\\
\qquad+\infty&\text{ if }\ \det\PP\ge0\,\end{cases}
\label{FH}\end{align}
with $\delta>0$ small  and $r$ big enough;
cf.\ \cite[Remark 2.6]{RouSte19FTCS}, \cite[Formula
(9.4.36)]{KruRou19MMCM},  or \cite{Neff}.

 Eventually, $\FG:\R^{d\times d}\to[0,\infty)$ 
controls the elastic strain gradient and relates to the length scale of
higher-order effects. Specific assumptions are given in \eqref{ass-G} below.
\UUU In particular, \EEE the stored energy \UUU features \EEE a regularizing
term \AAA depending on \EEE the gradient of the elastic strain $\UUU F_{\rm el} =
\nabla y \PP^{-1}$. \EEE Note however that no gradient
of $\PP$ appears in the energy, for this  might  give rise to hardening,
as explained in \UUU Subsection \ref{sec-motiv} below.\EEE



\UUU \subsection{Spurious hardening from gradients \TTT{in the stored energy}}\label{sec-motiv}

As already mentioned, the analysis of inelastic evolution models calls for
considering inelastic gradient theories. Usual choices in this
direction are terms of the form \EEE
\begin{subequations}\label{Phi}
\begin{align}
&&&&&\ \ \ \
\frac12\kappa|\nabla\PP|^2&&\quad\text{(standard choice)},\label{Phi1}&&&&&&\\
&&&&&\
\frac12\kappa|F^{-\top}\nabla\PP|^2&&\quad\text{(push forward)}, \label{Phi2}\\
&&&&&\frac12\kappa|\nabla(\PP^\top\PP)|^2\, &&\quad\text{(inelastic
metric tensor)}. \label{Phi3}
\end{align}
\end{subequations}
For the {\it standard choice} in \eqref{Phi1}, we refer to 
\cite{GMMM06ANMF,MaiMie09GERI,KruRou19MMCM,MieRou16RIEF} in the context
of plasticity, see also \cite{MiRoSa18GERV} for a more general dependence
on $\nabla\PP$ covering also creep models, as well as \cite{anand} for an
additional scalar-valued internal variable acting as an effective inelastic
strain. The {\it push-forward}  term in \eqref{Phi2}   has been used in
\cite[Remark 9.4.12]{KruRou19MMCM}  and  \cite[Remark 5]{Roub??CHEC}, whereas the
inelastic {\it metric} tensor in \eqref{Phi3} has been
analyzed in \cite{GraSte17FPQE}, cf.\ \cite{NefGhi16CIEP}\CHECK{ for
a throughout discussion and comparison}. 

All  models \eqref{Phi}  however exhibit a drawback: the influence of the inelastic
gradient terms amplifies when inelastic slips evolve and
accommodate large inelastic strains. This, in turn, \AAA might result \EEE in a 
spurious  hardening
effect.

 To demonstrate the presence of a non-autonomous spurious
hardening effect,  we consider $d=2$ and resort to
a stratified situation where $F$ and $\PP$ are constant in the $x_1$ direction,
cf.\ \cite{RouSte19FTCS} or also \cite[Example 9.4.11]{KruRou19MMCM} for
similar examples.
We consider a pure {\it horizontal} shift
of the stripe $\varOmega=\R{\times}[-\ell,\ell]$ 
 driven by time-dependent 
Dirichlet boundary conditions for the displacement
on the sides $\R{\times}\{\pm\ell\}$ 
 and  evolving in a steady-state mode.  In particular, we
assume by symmetry that the deformation has the {\it stratified} form
$$y(x_1,x_2) = (x_1+\UUU f(t,x_2),x_2)$$ 
where the slip via the (unspecified) smooth function \UUU $f
: [0,+\infty) \times [-\ell,\ell] \to \R$ \EEE  fulfills the given Dirichlet
boundary
conditions, say 
\begin{equation}
\UUU f \EEE (t,\pm \ell) = \pm t.\label{bcs}
\end{equation}
We  specify  elastic response by  assuming the material to
be rigid. In particular, the elastic strain $F_{\rm el}$
is assumed to be the identity matrix. In  the setting of plasticity, this would be called a
plastic-rigid model. The corresponding inelastic strain reads then 
\begin{align}\label{P-scaling}
\PP =  F&=\nabla
          y=\bigg(\!\!\begin{array}{cc}1\!&\!\UUU \partial_{x_2}f(t,x_2) \EEE
\\0\!&\!1\end{array}\!\!\bigg)\,.
\end{align}
Let us note that ${\rm det}\,\PP=1$, so that $\FH(\PP)=0$ when 
$\FH$ is defined as in \eqref{FH}.
The arguments in the $\kappa$-term in \eqref{Phi} read  (see
Section \ref{sec-model} for details on the tensorial notation)  then
as

\begin{subequations}\begin{align}
(\nabla\PP)_{ijk}&=
\left\{
  \begin{array}{ll}
   \UUU  \partial_{x_2}^2f(t,x_2) \EEE& \text{for} \ i=1, \, j=2, \, k=2,\\
0&\text{otherwise},
  \end{array}
\right., \\[1mm]
(F^{-\top}\nabla\PP)_{ijk}&=
\left\{
  \begin{array}{ll}
     \UUU  \partial_{x_2}^2f(t,x_2) \EEE& \text{for} \ i=1, \, j=2, \, k=2,\\
 -  \UUU  \partial_{x_2}f(t,x_2) \EEE   \UUU  \partial_{x_2}^2f(t,x_2) \EEE& \text{for} \ i= j= k=2,\\
0&\text{otherwise},
  \end{array}
\right.,\\[1mm]
(\nabla(\PP^\top\PP))_{ijk}&=
\left\{
  \begin{array}{ll}
   \UUU  \partial_{x_2}^2f(t,x_2) \EEE& \text{for} \ i=1, \, j=2, \, k=2,\\
   \UUU  \partial_{x_2}^2f(t,x_2) \EEE& \text{for} \ i=2, \, j=1, \, k=2,\\
 2   \UUU  \partial_{x_2}f(t,x_2) \EEE   \UUU  \partial_{x_2}^2f(t,x_2) \EEE& \text{for} \ i= j= k=2,\\
0&\text{otherwise},
  \end{array}
\right.\,.
\end{align}\end{subequations}
 
\UUU Note that $\partial_{x_2}f(t,x_2)$ necessarily depends on
time. Indeed, if this were not the case one would have that 
$$ \DT f(t,\ell)-\DT f(t,-\ell) =
\int_{-\ell}^\ell\partial_{x_2} \DT f(t,x_2)\, {\rm d} x_2
=0,$$
contradicting the fact that $\DT f(t,\pm\ell)=\pm 1$ from
\eqref{bcs}. Hence, in \EEE all cases, the argument of the quadratic
\UUU terms \EEE 
in \eqref{Phi} \UUU is genuinely time dependent. More precisely, by
taking the mean across the stripe we have that 
$$\frac{1}{2\ell}\int_{-\ell}^{\ell} \partial_{x_2} f(t,x_2)\, {\rm d} x_2 =\frac{1}{2\ell}\left(
f(t,\ell)-f(t,-\ell) \right)\stackrel{\eqref{bcs}}{=}\frac{t}{\ell}$$
so that the terms in \eqref{Phi} would actually be unbounded in time. \EEE
This shows, that no matter how small \UUU the coefficient $\kappa$ is, \EEE
\UUU the \EEE regularizing \UUU terms in \eqref{Phi} grow indefinitely
\EEE under large \EEE slips, \UUU preventing the energy from being bounded
and \EEE 
eventually  corrupting the modelization.  
To compensate for
these spurious hardening-like effects, one could assume $\kappa$ to be 
time dependent,  which would however  lead to an artificially non-autonomous 
model, which is also not desirable.

\UUU In order to avoid this spurious hardening effect while still
retaining regularization, our choice \eqref{free-energy+} for $\varPhi$
above departs from the classical inelastic-gradient regularization
\eqref{Phi} by including the gradient of the elastic
strain $F_{\rm el}$ instead. Note that in the above example the term
$\nabla F_{\rm el}$ vanishes, hence allowing for indefinitely large \UUU
inelastic slips under bounded energy.

Before closing this discussion, let us mention the possibility of
considering the alternative inelastic-gradient terms \EEE 
\begin{equation} 
\frac12 \kappa \big|{\rm curl} \PP\big|^2 \qquad \text{or} \qquad \frac12
\kappa \big|\PP^{-\top}{\rm curl}\PP\big|^2\label{disloc}
\end{equation}
\UUU in the energy $\varPhi$. \EEE
Here, the ${\rm curl}$
of the tensor $\PP$ is taken row-wise in three dimension and is
defined as ${\rm curl}\PP  = (\partial_1 \PP_{12} - \partial_2
\PP_{11}, \partial_1 \PP_{22} - \partial_2 \PP_{21})$ in two
dimensions. These terms correspond to the so-called
{\it dislocation-density} tensor \cite{Cermelli} and have been considered in
\cite{MielkeMueller,Neff,Scala} from the viewpoint of existence of solutions
\TTT of the incremental problems\EEE. In case of \eqref{P-scaling}, the
plastic strain is curl-free and both terms in \eqref{disloc}
vanish. \TTT Therefore these terms exhibit a capability 
\UUU
to accumulate large
\UUU inelastic slips at bounded energy, for they vanish for $\PP$ given by
\eqref{P-scaling}. In particular, \TTT at least in elastically
``well rigid'' materials, \UUU they would not generate the spurious hardening
effect mentioned above. \TTT However, the options \eqref{P-scaling}
do not seem \UUU  to contribute sufficient compactness in order to devise an
existence theory \AAA at the time-continuous level. \EEE
\TTT Of course, combination of some option from \eqref{disloc}
with some option from \eqref{free-energy+} in the stored energy
is possible and yields analytically good compactifying effects but again the
spurious hardening would be involved in the model.
\EEE

\UUU \subsection{Dissipation} \EEE

In order to incorporate inertial effects,
a Kelvin-Voigt-type viscosity needs to be included in the model. 
We consider a
purely linear viscous model by assuming the  dissipation
potential to be
quadratic in terms of rates, namely,
\begin{align}
&{\mathscr R}(y,\PP;\nabla \DT y ,\DT\PP)=
\int_\varOmega\frac{\nu_{\rm m}}2|\DT\PP|^2
+\frac{\nu_{\rm h}}2|\nabla^2\DT\PP|^2+\frac{\nu_{\rm kv}}2|\DT C_{\rm el}|^2
\,\d x\nonumber\\[-.3em]
&\qquad\qquad\qquad\qquad\qquad\text{ with }\ C_{\rm el}=F_{\rm el}^\top F_{\rm el}^{}=
\PP^{-\top}\nabla y^\top \nabla y\PP^{-1},\label{RRR}
\end{align}
where $\nu_{\rm m}$, $\nu_{\rm h}$, and $\nu_{\rm kv}$ are positive
viscous coefficients and
$C_{\rm el}$ is the elastic Cauchy-Green tensor. \UUU In particular,
the Kelvin-Voigt-type viscosity term depends on $\DT C_{\rm el}$ in
order to ensure frame-indifference \cite{Antm98PUVS}. \EEE

 The occurrence of the
$\nabla^2 \DT \PP$ term above is motivated by the need of controlling the rate of
$\PP$ uniformly in space while still 
 avoiding hardening.  In other words, differently from gradient terms acting directly
on $\PP$ (see \UUU Section \ref{sec-model}), \EEE this term provides a regularization
not giving rise to spurious hardening effects,  a
phenomenon which we want to avoid.
This uniform bound in space in turn will
allow the control of the nonlinear terms in \eqref{stresses} as well as of the
inverse $\PP^{-1}$, which is paramount for devising an existence
theory. Henceforth,  following a suggestion by A. Mielke  \cite{Miel??}, we  augment our dissipation potential 
by a regularization provided by the gradient of the creep rate.




The only higher-order terms involving the inelastic strain \UUU hence \EEE occur
in the dissipation and  are  given by 
 the gradient of  the inelastic strain rate,
i.e.\ of $\DT\PP$. 
\UUU With reference to the discussion of Subsection \ref{sec-motiv},
let us point out that such terms may again be time dependent. Still,
they can be expected to show some boundedness with respect to time. In the case
of \eqref{P-scaling} one indeed obtains that the mean across the strip
$$\frac{1}{2\ell}\int_{-\ell}^\ell \DT\Pi(t,x_1,x_2)\, {\rm d}x_2 = \frac{1}{2\ell}\int_{-\ell}^\ell
\bigg(\!\!\begin{array}{cc}0\!&\!\UUU  \partial_{x_2}\DT f(t,x_2) \EEE
\\0\!&\!0\end{array}\!\!\bigg) {\rm d}x_2 =
\bigg(\!\!\begin{array}{cc}0\!&\!\UUU  \frac{\DT f(t,\ell) - \DT f(t,-\ell)}{2\ell} \EEE
\\0\!&\!0\end{array}\!\!\bigg)\stackrel{\eqref{bcs}}{=} \bigg(\!\!\begin{array}{cc}0\!&\!\UUU 1/\ell \EEE
\\0\!&\!0\end{array}\!\!\bigg)$$
is time-independent. A regularization in term of $\nabla \DT \PP$ is
hence not expected to generate 
spurious hardening-like effects. 



\UUU \subsection{Constitutive equations}\EEE

Following the classical {\it Coleman-Noll procedure}
\cite{Coleman-Noll63}, we identify variations of $\varPhi$
with respect to $y$ and $\PP$
as driving
forces in the momentum equation and in the inelastic
flow-rule, respectively. More precisely, we have 
\begin{subequations}\label{stresses}\begin{align}
\delta_y\varPhi(y,\PP)&=-{\rm div}\left(\DD\FE(\nabla y\PP^{-1})\PP^{-\top}
    -{\rm div}\big(\DD\FG(\nabla(\nabla y\PP^{-1}))\big)\PP^{-\top}\right)\,,
\\
\delta_\PP\varPhi(y,\PP)&=\nabla y^\top\DD\FE(\nabla y\PP^{-1}){:}\DD(\PP^{-1})
  \nonumber\\
&\quad+\DD\FH(\PP)
-{\rm div}\big(\DD\FG(\nabla(\nabla y\PP^{-1}))\big){:}\nabla y \DD (\PP^{-1})\,.
\end{align}\end{subequations} 



\UUU In order to consider variations of the dissipation $\mathscr R$,
we start by explicitly computing \EEE   
\begin{align*}
\DT C_{\rm el}&=\PP^{-\top}(\nabla \DT y^\top \nabla y{+}\nabla
y^\top\nabla \DT y)\PP^{-1}
+(\DD(\PP^{-\top}){:}\DT\PP) \nabla y^\top \nabla y\PP^{-1}+\PP^{-\top}\nabla y^\top
  \nabla y\DD(\PP^{-1}){:}\DT\PP\\
&=\PP^{-\top}(\nabla \DT y^\top \nabla y{+}\nabla y^\top\nabla \DT y)\PP^{-1}
- \PP^{-\top}\DT\PP^\top \PP^{-\top} \nabla y^\top \nabla y\PP^{-1}-\PP^{-\top}\nabla y^\top \nabla y\PP^{-1}\DT \PP\PP^{-1}\\
&=\PP^{-\top}( \nabla \DT y^\top \nabla y{+}\nabla y^\top\nabla \DT y)\PP^{-1}
-2 \,{\rm sym}\,(\PP^{-\top}\nabla y^\top \nabla y \PP^{-1} \DT \PP \PP^{-1})\,.
\end{align*}
This Kelvin-Voigt-type viscosity features  then  both $\nabla \DT y$ and $\DT
\PP$ terms. It hence contributes to both the
momentum equation and to the inelastic flow rule.
In particular, setting for brevity $\varSigma:=\nu_{\rm kv} \DT C_{\rm el}$, the contribution of the
Kelvin-Voigt-type viscosity to the stress is given by
\begin{align*}
\delta_{\DT y}\DT C_{\rm el}{:}\varSigma=-{\rm div} \left(2\,{\rm
    sym}\,\big(\PP^{-\top}\nabla y^{\top}\varSigma\PP^{-1}\big) \right).
\end{align*}
On the other hand, by computing 
$$
\DD_{\DT \PP}\DT C_{\rm el}=\big(\PP^{-\top}\nabla y^{\top}\nabla y\,\DD(\PP^{-1})\big)^{\rm t}+
\PP^{-\top}\nabla y^{\top}\nabla y\,\DD(\PP^{-1})
\,,
$$
 we have that the Kelvin-Voigt-type viscous contribution to
 the inelastic driving force is
\begin{align*}
\DD_{\DT \PP}\DT C_{\rm el}{:}\varSigma=-2\,{\rm sym}\,\big(\PP^{-\top}\nabla y^{\top}\nabla y\PP^{-1}\varSigma\PP^{-1}\big)\,.
\end{align*}


\UUU \subsection{Evolution system} \EEE


The evolution of the medium is governed by the system of momentum equation and
the inelastic flow rule. Let us denote by  ${\mathscr T}(\DT y)=\frac12\int_\varOmega\varrho|\DT y|^2\,\d x$ the
kinetic energy and by ${\mathscr F}(t)$ the external load  
$$\langle{\mathscr F}(t),y\rangle=\int_\varOmega f(t){\cdot}y\,\d x
+\int_\varGamma g(t){\cdot}y\,\d S\,$$
where $f$ and $g$ denote a given body force density and surface
traction density, respectively. The system reads then in abstract form
\begin{subequations}\label{abstract}
  \begin{align}
    (\delta_{\DT y}{\mathscr T}(\DT y)){\hspace{-6.5mm} {\phantom{O}}^{\phantom{O}^\text{\LARGE.}}} +
    \delta_{\DT y}{\mathscr R}(y,\PP;\nabla \DT y,\DT\PP)
    +\delta_y\varPhi(y,\PP)&={\mathscr F}(t)\,,\label{abstract1}\\
    \delta_{ \DT\PP}{\mathscr R}(y,\PP;\nabla \DT y ,\DT\PP)
    +\delta_\PP\varPhi(y,\PP)&=0\,. \label{abstract2}
  \end{align}
\end{subequations}

Here, we have formally indicated variations with $\delta$. In the
following, these relations will be made precise in the weak sense, see
\eqref{weak-form}. For the sake of clarity, we present
here the strong form of the system, assuming suitable regularity of
the ingredients.
Owing to our choices
\eqref{free-energy+} and \eqref{RRR} for energy and dissipation, the
latter corresponds to the nonlinear PDE system 
\begin{subequations}\label{evol}
\begin{align}
    &\varrho\DDT y-{\rm div}\Big(\DD\FE(\nabla y\PP^{-1})\PP^{-\top} 
    +2\,{\rm sym}\,\big(\PP^{-\top}\nabla
    y^{\top}\varSigma\PP^{-1}\big) \Big)\nonumber\\
&
\label{evol1}
   \qquad\qquad\qquad\qquad  +{\rm div}\left({\rm
       div}\big(\DD\FG(\nabla(\nabla y\PP^{-1}))\big)\PP^{-\top} \right)
    =f,
 \\[3mm]&\nonumber 
\nu_{\rm m}\DT\PP+{\rm div}^2\big(\nu_{\rm h}\nabla^2\DT\PP\big)
    +\nabla y^\top\DD\FE(\nabla y\PP^{-1}){:}\DD(\PP^{-1})
    \nonumber\\
&\qquad\qquad\qquad\qquad -2\,{\rm sym}\,\big(\PP^{-\top}\nabla y^{\top}\nabla
     y\PP^{-1}\varSigma\PP^{-1}\big)+\DD\FH(\PP)
\nonumber\\&
\label{evol2}
\qquad\qquad\qquad \qquad
-{\rm div}\big(\DD\FG(\nabla(\nabla y\PP^{-1}))\big){:}\nabla y\DD(\PP^{-1})=0
\,,
\end{align}
\end{subequations}
where we have again used the notation
\begin{equation}
\varSigma=
{\nu_{\rm kv}} \DT C_{\rm el} \ \ \text{and} \ \ C_{\rm el}
=\PP^{-\top}\nabla y^\top\nabla y^\top\PP^{-1} \label{evol3}
\,.\end{equation}
Taking into account the formulas \eqref{green}--\eqref{greenS}, system
\eqref{evol} is intended to be completed 
by the following boundary conditions 
\begin{subequations}\label{BC}\begin{align}\nonumber
&\DD\FE(\nabla y\PP^{-1})\PP^{-\top}n
    -{\rm div}\big(\DD\FG(\nabla(\nabla y\PP^{-1}))\big)\PP^{-\top}\big)n  \nonumber 
\\&\hspace*{4em}-\divS\big(\DD\FG(\nabla(\nabla y\PP^{-1})\big)n\PP^{-\top}\big)
-2\mathfrak{h}\left(\DD\FG(\nabla(\nabla y\PP^{-1}))n\PP^{-\top}\right)n\nonumber
\\&\hspace*{4em}
  +2\,{\rm sym}\,\big(\PP^{-\top}\nabla y^{\top}\varSigma\PP^{-1}
  \big){n}=g \,,
\label{BC1}\\ & \big(\DD\FG(\nabla(\nabla y\PP^{-1})\big) \big){:}({n}\otimes({n}\PP^{-1}))=0\,,
\label{BC12}\\ &
\DD\FG(\nabla(\nabla y\PP^{-1})){n}{:}\nabla y\,\DD(\PP^{-1}) -{\rm div} \nu_{\rm h}\nabla^2\DT\PP n
\nonumber\\& \hspace*{4em}
-\divS(\nu_{\rm h}\nabla^2\DT\PP{n}) - 2 \nu_{\rm h}\mathfrak{h}(\nabla^2\DT\PP{n})n =0\,,\label{BC21}
\\& 
\nu_{\rm h}\nabla^2\DT \PP{:}({n}\otimes{n})=0\,.
\label{BC22} \end{align}\end{subequations}

The energetics of the model can be obtained by 
  formally testing \eqref{evol1} with $\DT y$ under
\eqref{BC1}--\eqref{BC12} and \eqref{evol2} with $\DT\PP$  under \eqref{BC21}--\eqref{BC22}.
 By considering the initial conditions
\begin{align}\label{IC}
y(0)=y_0,\ \ \ \DT y_0=v_0,\ \ \ \PP(0)=\PP_0\,,
\end{align}
the resulting energy balance on the time interval $[0,t]$ is
\begin{align}\nonumber
&\int_\varOmega\frac{\rho}{2} |\DT y(t)|^2
+
\FE(\nabla y(t)\PP^{-1}(t))
+
\FG(\nabla(\nabla y(t)\PP^{-1}(t)))
+
\FH(\PP(t))\,\d x
\\\nonumber
&\quad+\int_0^t\!\!\int_\varOmega
\nu_{\rm kv}|\DT C_{\rm el}|^2
+
\nu_{\rm m}|\DT\PP|^2 +\nu_{\rm h}|\nabla^2\DT \PP|^2
\,\d x\,d\tau
=\int_0^t\!\!\int_\varOmega f{\cdot}\DT y\,\d x\,d\tau
+\int_0^t\!\!\int_\varGamma g{\cdot}\DT y\,\d S\,d\tau
\\\label{energy}
&\qquad\qquad
+\int_\varOmega 
\frac{\rho}{2} |\DT y_0|^2
+\FE (\nabla y_0\PP^{-1}_0)
+
\FG(\nabla (\nabla y_0 \PP^{-1}_0))
+
\FH(\PP_0)\,\d x
\,.
\end{align}
%
%
In particular, the sum of total energy at time $t$ and dissipated energy on
$[0,t]$ equals the sum of initial total energy and work of external forces. 

\begin{remark}[{\sl Nonlinear or activated creep}]\label{rem-nonlin}\upshape
We assume here the dissipation potential to be quadratic, which makes
the occurrence of $\DT \PP$ in \eqref{evol2} linear. In order to
generalize this to the nonlinear
(or even activated) 
case, the analysis of the problem would require to check strong
compactness for the approximations of $\varSigma$. This  seems
presently out of reach in our setting. 
where only a weak convergence for such approximants can be guaranteed, cf.\ 
\eqref{first2} below. 
%
\end{remark}


\begin{remark}[{\sl  Jeffreys  rheology}]\label{rem-Jeffrey}\upshape
The combination of two viscous damping mechanisms and one elastic
energy-storing
mechanism is often referred to as {\it  Jeffreys  rheology} \cite{KruRou19MMCM}
(sometimes also called {\it anti-Zener}
rheology). This combination may arise from two different arrangements of
rheological elements: one can arrange a Stokes viscous element either in
parallel with a Maxwell
rheological element or in series with a
Kelvin-Voigt rheological one. Recall that a Maxwell (resp.\ Kelvin-Voigt)
rheological element is an arrangement of an elastic and a viscous element in series
(resp.\ in parallel). At small strains, the two possible
arrangements giving a  Jeffreys  rheology  
are equivalent, cf.\ \cite[Formula (6.6.34)]{KruRou19MMCM}.
On the contrary, equivalence does not hold at large strains. In our model we follow the
second variant: the viscous Stokes element is in series with a
Kelvin-Voigt rheological element. The reader is referred to \cite[Remark 9.4.4]{KruRou19MMCM},
for a model following the first variant instead, which allows for a
simpler analysis in spite of a somehow lesser physical
relevance.  
\end{remark}

\section{Analysis of the model
}\label{sec-anal}
 In the following we
use the standard notation $C(\cdot)$ for the space of continuous 
functions, 
$L^p$ for Lebesgue  spaces, and $W^{k,p}$ for Sobolev spaces whose 
$k$-th distributional derivatives are in $L^p$. Moreover, we use the 
abbreviation $H^k=W^{k,2}$ and,  for all $p\geq 1$, we let the conjugate
exponent $p'=p/(p{-}1)$  (with $p'=\infty$ if $p=1$), and use
the  notation 
$p^*$ for the Sobolev exponent $p^*=pd/(d{-}p)$ for $p<d$,
$p^*<\infty$ for $p=d$, and $p^*=\infty$ for $p>d$.
Thus, $W^{1,p}(\varOmega)\subset L^{p^*}\!(\varOmega)$ or 
$L^{{p^*}'}\!(\varOmega)\subset (W^{1,p}(\varOmega))^*$=\,the dual to $W^{1,p}(\varOmega)$. 

Given the fixed time interval $I=[0,T]$, we denote by $L^p(I;X)$ the 
standard Bochner space of Bochner-measurable mappings $u: I\to X$, where
$X$ is a Banach space. Moreover,  $W^{k,p}(I;X)$ denotes the Banach space of 
mappings in  $L^p(I;X)$ whose $k$-th distributional derivative in time is 
also in $L^p(I;X)$.

Let us list here the assumptions on the data which
are used in the following:
\begin{subequations}\label{ass}
\begin{align}\nonumber
&\FE
:\R^{d\times d}\to[0,+\infty]\ \text{ continuously differentiable
  on }\ {\rm GL}^+(d),  \ \exists\,\epsilon>0, \  p_{\rm G}\in
  (d,2^*), \ r> p_{\rm G}d/(p_{\rm G}-d), 
\\
&\qquad\FE(F_{\rm el})\ge\begin{cases}\epsilon
/(\det F_{\rm el})^r\!\!\!&\text{if }\ \det F_{\rm el}>0,
\\[-.2em]\quad+\infty&\text{if }\ \det F_{\rm el}\le0,\end{cases}\ \ \ \ 
,
\label{ass-FM}
\qquad
\\[-.2em]\nonumber
&\FH:\R^{d\times d}\to[0,+\infty]\ \text{ continuously differentiable
  on }\ {\rm GL}^+(d),  \ \exists\,\epsilon>0, \  s > 2^*d/(2^*-d),
   \\
&\qquad\FH(\PP)\ge\begin{cases}\epsilon
/(\det\PP)^s\!\!\!&\text{if }\ \det\PP>0,
\label{ass-plast-large-HD-growth}
\\[-.2em]\quad+\infty&\text{if }\ \det\PP\le0,\end{cases}\ \ \ \ 
,
\\[-.2em]\nonumber
&\FG:\R^{d\times d\times d}\to[0,+\infty)\ \text{ convex, continuously differentiable}, \ \exists\,\epsilon>0,
\\[-.2em]\nonumber
&\qquad\forall G,\widetilde G\in \R^{d\times d\times d}:\ \ \ (\DD\FG(G){-}\DD\FG(\widetilde G))\vdots(G{-}\widetilde G)\ge\epsilon|G{-}\widetilde G|^{\pG}
\\[-.2em]
&\hspace{11em}\FG(G)\ge\epsilon|G|^{\pG}\,,\ \ \
|\DD\FG(G)|\le(1+|G|^{\pG-1})/\epsilon\,,
\label{ass-G}\\
\label{ass-M-K}&\varrho>0,\ \ \nu_{\rm m},\nu_{\rm kv},\nu_{\rm h}>0,
\\&\nonumber
y_0\!\in\! W^{2,\pG}
(\varOmega)^d,\ \ 
v_0\!\in\! L^2(\varOmega)^d,\ \ 
\PP_0\!\in\! H^2
(\varOmega)^{d\times d},\ \ 
\\&\qquad
\label{ass-IC}\FE(\nabla y_0\PP_0^{-1})\!\in\!L^1(\varOmega),\ \ \FH(\PP_0)\!\in\!L^1(\varOmega),
\\&\label{ass-load}
f\in L^1(I;L^2(\varOmega)^d)+L^2(I;L^1(\varOmega)^d),\ \ \ g\in
L^2(I;L^1(\varGamma)^d).
\end{align}\end{subequations}
 A prototypical choice  
for $\FG$ satisfying \eqref{ass-G} is $\FG(\cdot)=|\cdot|^{\pG}$.
The restriction $\pG<2^*$ will be instrumental for estimates \eqref{nabla-Fel-strongly}
and  \eqref{first5} below.

The definition of weak solutions follows directly from system
\eqref{abstract}. It can be recovered by formally testing 
both equations in \eqref{evol}
by smooth functions and use Green formulas \eqref{green} together with
the surface Green formula \eqref{greenS}, the boundary conditions \eqref{BC}, and multiple
by-part integration in time, keeping into account the
initial conditions \eqref{IC}. 
Altogether,
we arrive at the following definition.


\begin{definition}[Weak formulation of 
\eqref{evol}  with  \eqref{BC}-\eqref{IC}]\label{def}
The pair 
$(y,\PP)$ satisfying
\begin{subequations}\label{weak-form-}\begin{align}\nonumber
&y\in L^\infty(I;W^{2,\pG}(\varOmega)^d)\cap H^1(I;L^2(\varOmega)^d)
\ \ \text{ with }\ \ \nabla y^\top\nabla y\in H^1(I;L^2(\varOmega)^{d\times d})\,,
\\[-.2em]&\qquad\qquad\varSigma\in L^2(I{\times}\varOmega)^{d\times d},\quad
\det\nabla y>0\,,\ \ \text{ and }\ \ \
 \frac1{\det\nabla y}\in L^\infty(I{\times}\varOmega)\,,
\label{weak-form-y}
\ \ \text{ and}
\\[-.2em]\label{weak-form-Pi}
 &\PP\in H^1(I;H^2(\varOmega)^{d\times d})
 \ \ \text{ with }\ \ \det\PP>0\ \ \text{ and }\ \ \
 \frac1{\det\PP}\in L^\infty(I{\times}\varOmega)
\end{align}\end{subequations}
is called a \emph{weak solution} to the initial-boundary-value problem
\eqref{evol}, \eqref{BC}--\eqref{IC} if the following two identities
hold with $\varSigma$ from \eqref{evol3}:
\begin{itemize}
\item[\rm (i)]
The  \emph{weak formulation of  the momentum balance} \eqref{evol1} 
 with the boundary conditions \eqref{BC1}--\eqref{BC12} and first two
 initial conditions in
 \eqref{IC}
 \begin{subequations}
\label{weak-form}\begin{align}\nonumber
&
\int_0^T\!\!\int_\varOmega\Big(\DD\FE(\nabla y\PP^{-1}){:}(\nabla\widetilde y\,\PP^{-1})
+\varrho y{\cdot}\DDT{\widetilde y}
+2\,{\rm sym}\,\big(\PP^{-\top}\nabla y^{\top}\varSigma\PP^{-1}\big):\nabla\widetilde y
\\[-.4em]&\nonumber\qquad\qquad
+\DD\FG(\nabla(\nabla y\PP^{-1})){\vdots}\nabla(\nabla\widetilde y\PP^{-1})
\Big)\,\d x\, \d t
=\int_0^T\!\!\int_\varOmega\!
f{\cdot}\widetilde y\,\d x\, \d t
\\[-.4em]&\qquad\qquad\qquad\qquad+\int_0^T\!\!\int_{\varGamma}\!
g{\cdot}\widetilde y\,\d S\, \d t
+\int_\varOmega\!\varrho
v_0{\cdot}\widetilde y(0)-\varrho y_0{\cdot}\DT{\widetilde y}(0)\,\d x
\label{momentum-weak}\end{align}
holds for any $\widetilde y$ smooth with $\widetilde y(T)=\DT{\widetilde y}(T) =0$.
\item[\rm (ii)] The \emph{weak formulation of the
creep flow rule} \eqref{evol2} with the boundary conditions \eqref{BC21}--\eqref{BC22} and the
last initial condition in \eqref{IC}
\begin{align}\nonumber
&\int_0^T\!\!\int_\varOmega\Big(
\nabla y^\top\DD\FE(\nabla y\PP^{-1}){:}\DD(\PP^{-1})+\DD\FH(\PP)
-2\,{\rm sym}\,\big(\PP^{-\top}\nabla y^{\top}\nabla y\PP^{-1}\varSigma\PP^{-1}\big)
\Big)
{:}{\widetilde\PP}
\\[-.5em]&\nonumber\qquad
-
\nu_{\rm m}\PP:\DT{\widetilde\PP}
+\DD\FG(\nabla(\nabla y\PP^{-1}))\vdots\nabla\big(\nabla y\DD(\PP^{-1}){:}{\widetilde\PP}\big)
-\nu_{\rm h}\nabla^2\PP\vdots\nabla^2\DT{\widetilde\PP}
\,\d x\, \d t
\\[-.0em]&\qquad\qquad=\int_\varOmega
\nu_{\rm m}\PP_0{:}{\widetilde\PP}(0)
+\nu_{\rm h}\nabla^2\PP_0\vdots\nabla^2{\widetilde\PP}(0)\,\d x
\label{weak-form-P}
\end{align}\end{subequations}
holds for any ${\widetilde\PP}$ smooth with ${\widetilde\PP}(T)=0$.
\end{itemize}
\end{definition}

Let us note that, due to \eqref{weak-form-Pi},
we have
also $\PP^{-1}=\Cof\PP^\top/\det\PP
\in L^\infty(I{\times}\varOmega)^{d\times d}$, \CHECK{as well as $\DD\FG(\nabla(\nabla y\PP^{-1}))\in L^\infty(I;L^{\pG'}(\varOmega)^{d\times d\times d})$} 
so that
all integrands in \eqref{weak-form} are well-defined as $L^1$-functions.



 Our main analytical result is an existence theorem for weak
solutions.  This is to  be seen as a mathematical consistency property of
the proposed model. It reads as follows. 

\begin{theorem}[Existence of weak solutions]\label{thm}
Let the assumptions \eqref{ass} hold. Then,  there exists  a weak solution 
$(y,\PP)$ in the sense of Definition~{\rm {\ref{def}}}. 

\end{theorem}

\begin{proof}
As we are working in reference (Lagrangian) coordinates and aim at 
testing by partial derivatives in time, we can advantageously
use the Galerkin discretisation method in space. Let us fix a nested sequence
of finite-dimensional subspaces $V_k\subset W^{2,\infty}(\varOmega)$,
$k\in\N$ whose union is dense in $W^{2,\infty}(\varOmega)$.
We will use this
sequence for  all components of deformations $y$ and inelastic strains
$\PP$.  

Without loss of generality, we may consider an approximation of the initial
conditions $y_{0,k}\in V_k^d$, $v_{0,k}\in V_k^d$, and $\PP_{0,k}\in V_k^{d\times d}$
such that 
\begin{subequations}\label{IC-approx}\begin{align}
&&&&&y_{0,k}\to y_0&&\text{strongly in }W^{2,\pG}(\varOmega)^d,&&&&&&&&
\\
&&&&&v_{0,k}\to v_0&&\text{strongly in }L^2(\varOmega)^d,
\\
&&&&&\PP_{0,k}\to\PP_0&&\text{strongly in }H^2(\varOmega)^{d\times d}.
\end{align}\end{subequations}
Existence of a finite-dimensional approximate solution
$(y_k,\PP_k)\in W^{2,1}(I;V_k^d)\times C^1(I;V_k^{d\times d})$
of the initial-value problem for the system of nonlinear ordinary 
differential \CHECK{equations} 
arising from the Galerkin approximation is
standard, also using successive prolongation
based on uniform $L^\infty$ estimates. Such estimates can be obtained
by testing the discrete-in-space equations by 
$\DT y_k$ and $\DT\PP_k$. This leads to 
the energy balance \eqref{energy} for the Galerkin approximations
$(y_k,\PP_k)$. Starting from the energy balance, by using the Gronwall and H\"older inequalities,
we obtain a-priori estimates independently of $k$, namely,
\begin{subequations}\label{est}\begin{align}\label{est1}
&\{y_k\}_{k\in\N}^{}\ \ \text{ is bounded in }\ W^{1,\infty}(I;L^2(\varOmega)^d),
\\\label{est2}
&\{\PP_k\}_{k\in\N}^{}\ \ \text{ is bounded in }\  H^1(I;H^2(\varOmega)^{d\times d})\subset
L^\infty(I{\times}\varOmega)^{d\times d}
,\\\label{est3}
&\{F_{{\rm el},k}\}_{k\in\N}^{}=\{\nabla y_k \PP^{-1}_k\}_{k\in\N}^{}\ \
  \text{ is bounded in }\  L^\infty(I;W^{1,\pG}(\varOmega)^{d\times d}),\\\label{est4}
&\{C_{{\rm el},k}\}_{k\in\N}^{}=\{F_{{\rm el},k}^\top F_{\rm
  el,k}\}_{k\in\N}^{}\ \ \text{ is bounded in }\  H^1(I;L^2(\varOmega)^{d\times d})
.
%
\intertext{Next, we use the 
classical Healey-Kr\"omer \cite{HeaKro09IWSS} argument, here
applied to the plastic strain 
instead of the deformation gradient, as already exploited in
\cite{RouSte19FTCS}. 
This is based on the $L^\infty$-bound of $\PP_k$
and on the sufficiently fast blow-up of $\FH$, as assumed in
\eqref{ass-plast-large-HD-growth}. It is important that the argument in
\cite{HeaKro09IWSS} holds even for the discrete level
(as realized already in \cite{KruRou19MMCM,MieRou20TKVR})
and ensures that $\det\PP_k\ge\delta$ for all time instants and for
some $\delta>0$ independent of $k$. In particular, we also have that}
&\label{P-1-bound}
\{\PP_k^{-1}\}_{k\in\N}^{}\ \ \text{ is bounded in }\ L^\infty(I\times \varOmega)^{d\times d}.
\intertext{From (\ref{est2})--(\ref{est3}) we get that
$\{\nabla y_k\}_{k\in\N}^{}=\{F_{{\rm el},k}\PP_k\}_{k\in\N}^{}$ is bounded in
$L^\infty(I\times \varOmega)^{d\times d\times d}$.
From \eqref{est2} we find that $\nabla(\nabla y_k\PP_k^{-1})
=
(\PP_k^{-\top}\nabla(\nabla y_k)^\top)^{\rm t}+
\nabla y_k\DD(\PP_k^{-1})\nabla\PP_k
$ is bounded in 
$L^\infty(I;L^{\pG}(\varOmega)^{d\times d\times d})$. This in particular implies that}
&
\{
\nabla(\nabla y_k)^\top\}_{k\in\N}^{}=\Big\{\PP_k^\top\Big(\nabla(\nabla y_k\PP_k^{-1}){-}\nabla y_k\DD(\PP_k^{-1})\nabla\PP_k\Big)^{\rm t}\Big\}_{k\in\N}^{}\nonumber\\
&\qquad\qquad\qquad\qquad\qquad\text{ is bounded in }\ L^\infty(I;L^{\pG}(\varOmega)^{d\times d\times d}).\label{est_last}
\intertext{From \eqref{est1}, we know that $\{y_k\}_{k\in\N}^{}$ is bounded in
$L^\infty(I;L^2(\varOmega)^d)$, so that 
\eqref{est_last}
 yields  
a bound in $L^\infty(I;W^{2,\pG}(\varOmega)^d)$.
We proceed by showing that \eqref{est4}, yields the estimate}
&\{\nabla\DT y_k\}_{k\in\N}^{}
\text{ is bounded in }\ L^2(I{\times}\varOmega)^{d\times d}. \label{est-of-DT-nabla-y}
\end{align}
\end{subequations}

To prove \eqref{est-of-DT-nabla-y} we argue as in \cite[Sect.\,9.4.3]{KruRou19MMCM}. We preliminary observe that by the growth conditions from below on $\FE$ in \eqref{ass-FM}, as well as by the super-quadratic growth on $\FG$ in \eqref{ass-G}, the Healey-Kr\"omer argument yields the existence of $\delta_{\rm el}>0$ such that

$${\rm det}\, F_{{\rm el},k}\geq \delta_{\rm el}\quad\text{ in }I\times \varOmega$$
for every $k\in \mathbb{N}$. By combining  the  Cauchy-Binet formula with the bound in \eqref{P-1-bound}, we find that 
$$
\frac1{\det\nabla y_k}=\frac1{\det(\nabla y_k\PP_k^{-1}\PP_k)}
=\frac1{\det(\nabla y_k\PP_k^{-1})}\frac1{\det\PP_k}\,
$$
is uniformly bounded in $L^\infty(I\times \varOmega)$. Property \eqref{est-of-DT-nabla-y} follows now by applying  
the generalized Korn inequality by Neff \cite{Neff02KFIN}
and Pompe \cite{Pomp03KFIV} as exploited for the Kelvin-Voigt rheology in
\cite[Thm.\,3.3]{MieRou20TKVR}. 

For all $k \in {\mathbb N}$ the pair $(y_k,\PP_k)$ fulfills the weak formulation
\eqref{weak-form}
with initial conditions approximated as \eqref{IC-approx}, 
provided that the test-functions take value in the finite-dimensional
space. In particular, we
have 
 \begin{subequations}\label{weak-form-k}\begin{align}\nonumber
&
\int_0^T\!\!\int_\varOmega\Big(\DD\FE(\nabla y_k\PP^{-1}_k){:}(\nabla\widetilde y_k\,\PP^{-1}_k)
+\varrho y_k{\cdot}\DDT{\widetilde y}_k
+2\,{\rm sym}\,\big(\PP^{-\top}_k\nabla y_k^{\top}\varSigma_k\PP_k^{-1}\big){:}\nabla\widetilde y_k
\\[-.4em]&\nonumber\qquad
+\DD\FG(\nabla(\nabla y_k\PP_k^{-1})){\vdots}\nabla(\nabla\widetilde y_k\PP^{-1}_k)
\Big)\,\d x\, \d t 
=\int_0^T\!\!\int_\varOmega\!
f{\cdot}\widetilde y_k\,\d x\, \d t
\\[-.4em]&\qquad\qquad+\int_0^T\!\!\int_{\varGamma}\!
g{\cdot}\widetilde y_k\,\d S\, \d t
+\int_\varOmega\!\varrho
v_0{\cdot}\widetilde y_k(0)-\varrho y_0{\cdot}\DT{\widetilde y}_k(0)\,\d x
\label{momentum-weak-k}\\\nonumber
&\int_0^T\!\!\int_\varOmega\Big(
\nabla y^\top_k\DD\FE(\nabla y_k\PP^{-1}_k){:}\DD(\PP^{-1}_k)+\DD\FH(\PP_k)
-2\,{\rm sym}\,\big(\PP^{-\top}_k\nabla y^{\top}_k\nabla y_k\PP^{-1}_k\varSigma_k\PP^{-1}_k\big)
\Big)
{:}{\widetilde\PP_k}
\\[-.5em]&\nonumber\qquad
- 
\nu_{\rm m}\PP_k{:}\DT{\widetilde\PP}_k
+\DD\FG(\nabla(\nabla y_k\PP^{-1}_k)){\vdots}\nabla\big(\nabla y_k\DD(\PP^{-1}_k):{\widetilde\PP}_k\big)
-\nu_{\rm h}\nabla^2\PP_k{\vdots}\nabla^2\DT{\widetilde\PP}_k
\,\d x\, \d t
\\[-.0em]&\qquad\qquad=\int_\varOmega
\nu_{\rm m}\PP_0{:}{\widetilde\PP}_k(0)
+\nu_{\rm h}\nabla^2\PP_0{\vdots}\nabla^2{\widetilde\PP}_k(0)\,\d x
\label{weak-form-P-k}
\end{align}\end{subequations}
for all $\widetilde y_k\in C^2(I;V_k^d)$
and $\widetilde\PP_k\in C^1(I;V_k^{d\times d})$ with   $\widetilde
y_k(T)=\DT{\widetilde y}_k(T) =0$ and $\PP_k(T)=0$.


We are hence ready to address the convergence $\{(y_k,\PP_k)\}_{k \in
  {\mathbb N}}$ as $k\to\infty$. By the Banach selection principle and
the Aubin-Lions compact-embedding theorem, we select a not
relabeled subsequence 
converging with respect to the weak* topologies indicated in 
\eqref{est}. In particular, we have that 
\begin{subequations}\label{conv}
\begin{align}
&y_k \to y \quad \text{weakly* in}\ W^{1,\infty}(I;L^2(\varOmega)^d)\cap
  L^\infty(I;W^{2,\pG}(\varOmega)^d) \nonumber\\
&\qquad\qquad\qquad\qquad\qquad\qquad\qquad 
\text{and strongly in}\ C(I{\times}\bar\varOmega)^d 
, \label{conv1}
\\
\label{conv2}
&\PP_k\to \PP \quad \text{weakly in}\  H^1(I;H^2(\varOmega)^{d\times d}) \ \text{and strongly in}\ L^\infty(I{\times}\varOmega)^{d\times d},
\\
\label{conv3}
&\PP_k^{-1}\to \PP^{-1} \quad   \text{strongly in}\ L^\infty(I{\times}\varOmega)^{d\times d},\\
\label{conv4}
&F_{{\rm el},k}=\nabla y_k \PP^{-1}_k\to F_{\rm el} =\nabla y
  \PP^{-1}\quad \text{weakly* in}\ L^\infty(I;W^{1,\pG}(\varOmega)^{d\times
  d}),\\
\label{conv5}
&C_{{\rm el},k} =F_{{\rm el},k}^\top F_{\rm
  el,k} \to C_{\rm el} =F_{\rm el}^\top F_{\rm
  el} \quad \text{weakly in}\  H^1(I;L^2(\varOmega)^{d\times d}).
\end{align}
\end{subequations}
In fact, using the Aubin-Lions theorem in the context of Galerkin method
when the time derivatives are estimated only in  some  locally convex space
(or alternatively only their Hahn-Banach extension is estimated in a Banach
space) requires some attention, as commented in
\cite[Sect.8.4]{Roub13NPDE}. The convergence of $\PP^{-1}_k$ is
obtained by exploiting the formula $\PP^{-1}_k = {\rm Cof}
\PP_k^\top/{\rm det} \PP_k$, 
as well as the uniform lower bound ${\rm det}\,\PP_k\geq\delta$, and the fact that the determinant is a locally Lipschitz function. By recalling that $\DD(\PP^{-1}){:}A = - \PP^{-1}A \PP^{-1}$
for all $A \in \R^{d\times d}$
one readily checks that 
\begin{align}
  \label{strong2}
&  \DD(\PP^{-1}_k) \to \DD(\PP^{-1}) \ \ \text{strongly in} \ \
  L^\infty(I{\times}\varOmega)^{d\times d\times d\times d },\\
&\nabla (\PP^{-1}_k)= \DD(\PP^{-1}_k){:}\nabla \PP_k \to
  \DD(\PP^{-1}){:}\nabla \PP = \nabla (\PP^{-1}) \nonumber\\
&\label{strong3}\qquad\qquad\qquad \text{strongly in} \ \
  L^q(I{\times}\varOmega)^{d\times d\times d } \quad \forall
  q<2^*.
\end{align}
We further proceed by proving that
\begin{align}
  \label{strong4}
\nabla (\nabla y_k^{}\PP^{-1}_k)\to\nabla (\nabla y\PP^{-1})\qquad\qquad\qquad \text{strongly in} \ \
L^{\pG}(I{\times}\varOmega)^{d\times d\times d}.
\end{align}
By the uniform monotonicity
of $\DD\FG$, we find:
\begin{align}\nonumber
&\nonumber\epsilon
\big\|\nabla(\nabla y_k\PP^{-1}_k)
-
\nabla (\nabla y\PP^{-1})
\big\|_{L^{\pG}(I{\times}\varOmega)^{d\times d\times d}}^{\pG}
\\
&\nonumber\quad\le
\int_0^T\!\!\!\int_\varOmega
\big(\DD\FG(\nabla(\nabla y_k\PP^{-1}_k))-\DD\FG(\nabla(\nabla y\PP^{-1}))\big)\vdots
\big(\nabla(\nabla y_k\PP^{-1}_k)-\nabla(\nabla y\PP^{-1}))\big)\,\d x\,\d t
\\
&\nonumber\quad=
\int_0^T\!\!\!\int_\varOmega
\DD\FG(\nabla(\nabla y_k\PP_k^{-1}))\vdots\nabla (\nabla(y_k{-}y)\PP^{-1}_k)\,\d x\,\d t
\\[-.2em]&\nonumber\quad\qquad+
\int_0^T\!\!\!\int_{\varOmega}
\DD\FG(\nabla(\nabla y_k\PP_k^{-1}))\vdots\nabla(\nabla y(\PP^{-1}_k{-}\PP^{-1}))\,\d x\, \d t
\\[-.2em]&\nonumber\quad\qquad\qquad
-
\int_0^T\!\!\!\int_{\varOmega}\!\DD\FG(\nabla (\nabla y\PP^{-1}))\vdots\big(\nabla (\nabla y_k\PP^{-1}_k){-}\nabla (\nabla y\PP^{-1})\big)\,\d x\d t
\ =:\ I_{1,k}+I_{2,k}+I_{3,k}\,.
\end{align}
where $\epsilon>0$ is from \eqref{ass-G}.
We have $I_{3,k}\to0$ owing to \eqref{conv4} and to
$\DD\FG(\nabla (\nabla y\PP^{-1}))\in L^\infty(I;L^{\pG'}(\varOmega)^{d\times d\times d})$
because of the growth assumption in \eqref{ass-G}. Also $I_{2,k}\to0$ since 
\begin{align}
\nabla (\nabla y(\PP^{-1}_k{-}\PP^{-1}))
=
((\PP^{-\top}_k{-}\PP^{-\top})\nabla(\nabla y)^\top)^{\rm t}
+\nabla y(\nabla\PP^{-1}_k{-}\nabla\PP^{-1})
\to0
\label{nabla-Fel-strongly}\end{align}
strongly in $L^{\pG}(I{\times}\varOmega)^{d\times d\times d}$ by \eqref{conv3};
here we  also used the convergence  $\nabla\PP^{-1}_k\to\nabla\PP^{-1}$ strongly in
$L^{\pG}(I{\times}\varOmega)^{d\times d\times d}$ owing to
\eqref{conv2}. 
To prove that also the term $I_{1,k}$ converges to 0, we test the momentum equation
for the Galerkin approximants by $y_k{-}\widetilde y_k$
where $\widetilde y_k$ is an approximation of the limit $y$ which takes values
in the finite-dimensional subspaces $V_k^d$ and which converges to $y$
in $L^2(I;W^{2,\pG}(\varOmega)^d)\cap H^1(I;L^2(\varOmega)^d)$.
We further assume
$\widetilde y_k(0)=y_{0,k}$.
Note that here $y_k{-}\widetilde{y}_k$ is not
$C^2(I;H^2(\Omega)^d)$ but rather
$W^{1,\infty}(I;H^2(\Omega)^d)$. Nevertheless, this regularity is
enough for arguing differently from \eqref{momentum-weak-k} and
integrating by-part in time only once. Note also that
$y_k(T)-\widetilde{y}_k(T)\neq 0$. For this reason, a further 
term  
at time $T$ appears in the equation below. Altogether, 
\begin{align}\nonumber
I_{1,k}&=
\int_0^T\!\!\int_\varOmega
\DD\FG(\nabla(\nabla y_k\PP_k^{-1}))\vdots
(\nabla (\nabla(\widetilde y_k{-}y)\PP^{-1}_k))
+\varrho\DT y_k{\cdot}(\DT y_k{-}\DT{\widetilde y}_k)
+f{\cdot}(y_k{-}\widetilde y_k)
\\[-.3em]&\nonumber\qquad\quad
-\DD\FE(\nabla y_k\PP^{-1}_k){:}\nabla(y_k{-}\widetilde y_k)\PP^{-1}_k)
-2\,{\rm sym}\,\big(\PP^{-\top}_k\nabla y_k^{\top}\varSigma_k\PP_k^{-1}\big){:}\nabla(y_k{-}\widetilde y_k)\,\d x\, \d t 
\\[-.1em]&\nonumber\qquad\qquad\qquad\qquad
-\int_\varOmega
\DT y_k(T){\cdot}(y_k(T){-}\widetilde y_k(T))\,\d x
+\int_0^T\!\!\int_{\varGamma}\! g{\cdot}(y_k{-}\widetilde y_k)\,\d S\, \d t\to0\,.
\end{align}

Then, from \eqref{est-of-DT-nabla-y}, using (the above mentioned generalization
of) the Aubin-Lions theorem, exploiting an information about $\DDT y_k$ obtained
via a comparison argument in the discrete variant of \eqref{evol1} for the
Galerkin approximants, we infer that
$$
\DT y_k\to \DT y\quad\text{ strongly in }L^2(I\times \varOmega)^d,
$$
and
$$
\DT y_k(T)\to \DT y(T)\quad\text{ weakly in }L^2( \varOmega )^d.
$$
By \eqref{est2}, \eqref{conv3}, \eqref{conv4}, and \eqref{conv5} we conclude that $I_{1,k}\to 0$ and obtain \eqref{strong4}.

What it is left to prove is that $(y,\PP)$ is a weak solution in the
sense of Definition \ref{def}. Let $\widetilde y$ and $\widetilde \PP$ be
smooth with $\widetilde y(T)=\DT{\widetilde y}(T)=0$ and $\widetilde{\PP}(T)=0$,
and approximate them via sequences
$\widetilde{y}_k$ and $\widetilde{\PP}_k$ as in \eqref{weak-form-k}, so that
$\widetilde{y}_k \to\widetilde y$ strongly in
$H^2(I;W^{2,\pG}(\varOmega)^d)$ and $\widetilde{\PP}_k \to  \widetilde{\PP}$ strongly in
$H^1(I;H^2(\varOmega)^d)$. One needs to
check that convergences \eqref{conv} are sufficient to pass to the
limit in all terms in \eqref{weak-form}. Let us start by the
momentum balance \eqref{momentum-weak-k}. By 
the continuity of the superposition operator
we have that 
\begin{align}
&\DD\FE(\nabla y_k\PP_k^{-1})\PP^{-\top}_k\to\DD\FE(\nabla y\PP^{-1}) \PP^{-\top}\text{ strongly in}
 \ \ L^\infty(I{\times}\varOmega)^{d\times d}\,,
\label{first1}
\end{align} 
cf.\ the growth condition \eqref{ass-FM}. Estimate
\eqref{est4} ensures that 
\begin{equation}\label{first2}\varSigma_k=\nu_{\rm kv}\DT C_{{\rm el},k}
\to \varSigma \ \ \text{weakly in} \ L^2(I{\times}\varOmega)^{d\times d}.
\end{equation} 
The limit $\varSigma$ can be identified
as $\varSigma=\nu_{\rm kv}\DT C_{\rm el}$ since we have convergence \eqref{conv5}. Owing to \eqref{first2}, \eqref{est_last}, and \eqref{conv3} we 
deduce that 
\begin{align}
&\PP^{-\top}_k \nabla y_k^\top   \varSigma_k
  \PP^{-1}_k \to  \PP^{-\top}\nabla y^\top  \varSigma
  \PP^{-1} \ \ \text{weakly in} \
L^2(I{\times}\varOmega)^{d\times d}\,. 
\label{first3}
\end{align}
Let us now compute 
\begin{align*}
\DD\FG(\nabla (\nabla y_k \PP^{-1}_k)){\vdots} \nabla (\nabla \widetilde y_k
\PP^{-1}_k)&= \DD\FG(\nabla (\nabla y_k \PP^{-1}_k)){\vdots}
(\PP^{-\top}_k\nabla(\nabla\widetilde y_k)^\top)^{\rm t}
\\[-.3em]&\ \ \ +\DD\FG(\nabla (\nabla y_k \PP^{-1}_k)){\vdots}  (\nabla \widetilde
  y_k \DD(
\PP^{-1}_k){:}\nabla \PP_k)
\end{align*}
Convergences \eqref{conv} 
suffice to pass to the weak limit in both terms
in the right-hand side.
In fact, taking into account \eqref{strong2} and \eqref{strong4}, we
have the following strong convergences (even though weak ones 
would  be enough
for our existence proof):
\begin{align}
 & \DD\FG(\nabla(\nabla y_k\PP^{-1}_k)){\vdots}
(\PP^{-\top}_k\nabla(\nabla\widetilde y_k)^\top)^{\rm t}\to
\DD\FG(\nabla (\nabla y \PP^{-1})){\vdots}
(\PP^{-\top}\nabla(\nabla\widetilde y)^\top)^{\rm t}
\nonumber\\
&\hspace*{14em}  \text{
in} \
  L^p(I;L^{\pG'}(\varOmega)^{d\times d})\qquad  \forall p<+\infty\,,\ \text{ and}
\label{first4}\\
& \DD\FG(\nabla (\nabla y_k \PP^{-1}_k)){\vdots}  (\nabla \widetilde
  y_k \,\DD(
\PP^{-1}_k){:}\nabla \PP_k)\to \DD\FG(\nabla (\nabla y \PP^{-1})){\vdots}  (\nabla \widetilde
  y \,\DD(
\PP^{-1}){:}\nabla \PP)  \nonumber\\
&\hspace*{14em} \text{
in} \
  L^p(I;L^q(\varOmega)^{d\times d})\qquad  \forall p<
  +\infty, \ q<\frac{2^*\pG'}{2^*{+}\pG'}.\label{first5} 
\end{align}

Since all  the  remaining terms in the momentum balance
\eqref{momentum-weak-k} are linear, convergences
\eqref{first1}--\eqref{first5} allow to pass to the limit and obtain \eqref{momentum-weak}.

Let us now move to the flow rule \eqref{weak-form-P-k}. Arguing
as above, by \eqref{est_last}
we have that 
\begin{align}
& \nabla y^\top_k \DD \FE(\nabla y_k \PP^{-1}_k){:}\DD (\PP^{-1}_k)
  \to \nabla y^\top \DD \FE(\nabla y \PP^{-1}){:}\DD (\PP^{-1})
  \nonumber\\
&\qquad \ \text{ strongly in } L^\infty(I{\times}\varOmega)^{d\times d}\,.
  \label{second1}  
\end{align} 
By using again convergence \eqref{first2} we also get that 
\begin{align}
& \PP^{-\top}_k \nabla y_k^\top \nabla y_k \PP^{-1}_k \varSigma_k
  \PP^{-1}_k \to  \PP^{-\top}\nabla y^\top \nabla y \PP^{-1}\varSigma
  \PP^{-1}\quad \text{weakly in} \
L^2(I{\times}\varOmega)^{d\times d}\,.
  \label{second2}
\end{align}
Eventually, we use convergences \eqref{conv1}, \eqref{strong3}, and \eqref{strong4} in
order to check that 
\begin{align*}
 &\DD\FG(\nabla(\nabla y_k\PP^{-1}_k)){\vdots}\nabla\big(\nabla
  y_k\DD(\PP^{-1}_k){:}{\widetilde\PP}_k\big) =
  -\DD\FG(\nabla(\nabla y_k\PP^{-1}_k)){\vdots}\nabla\big(\nabla
  y_k \PP^{-1}_k {\widetilde\PP}_k \PP^{-1}_k \big)\nonumber\\
&= - \DD\FG(\nabla(\nabla y_k\PP^{-1}_k)){\vdots}\big(
[(\widetilde{\PP}_k\PP_k^{-1})^{\top}\nabla(\nabla y_k\PP^{-1}_k)^t]^t+ \nabla y_k\PP^{-1}_k\nabla (\widetilde{\PP}_k\PP_k^{-1})\big)\nonumber\\
&\qquad \to \DD\FG(\nabla(\nabla y\PP^{-1})){\vdots}\nabla\big(\nabla
  y\DD(\PP^{-1}){:}{\widetilde\PP}\big) \ \ \text{strongly in} \
  L^1(I{\times}\varOmega)^{d\times d\times d}.
\end{align*}
All remaining terms in the flow rule
\eqref{weak-form-P-k} are linear and convergences
\eqref{second1}--\eqref{second2} suffice to pass to the limit and
obtain \eqref{weak-form-P}.
\end{proof}

\section*{Acknowledgments}
This research has been partially supported also from the
CSF (Czech Science Foundation) project 19-04956S, 
the M\v SMT \v CR
(Ministry of Education of the Czech Rep.) project
CZ.02.1.01/0.0/0.0/15-003/0000493,
the Austria Science Fund (FWF) projects F\,65, I\,2375,
P\,27052, I\,4052, and V\,662, and by the Vienna Science and Technology Fund (WWTF)
through Project MA14-009
as well as from BMBWF through the OeAD-WTZ project CZ04/2019
and the institutional support RVO: 61388998 (\v CR).
Besides, T.R.\ is thankful for the hospitality and support of the University
of Vienna. 

\bibliographystyle{alpha}

\end{document}